\documentclass[12pt]{article}
\usepackage[top=2.25cm, bottom=2.25cm, left=2.2cm, right=2.2cm]{geometry}
\usepackage{amsmath,amssymb,amsthm,mathtools}
\usepackage{microtype}
\usepackage{enumitem}
\usepackage{booktabs}
\usepackage[hidelinks]{hyperref}
\usepackage[nameinlink,capitalize]{cleveref}
\usepackage{xcolor}

\newtheorem{theorem}{Theorem}[section]
\newtheorem{lemma}[theorem]{Lemma}

\newtheorem{corollary}[theorem]{Corollary}
\newtheorem{claim}[theorem]{Claim}

\newtheorem{conjecture}[theorem]{Conjecture}
\theoremstyle{remark}

\newcommand{\R}{\mathbb R}

\newcommand{\cS}{\mathcal S}
\newcommand{\cK}{\mathcal K}

\newcommand{\eps}{\varepsilon}

\begin{document}
\title{The \(O(m^{2/3})\) error term for judicious partitions of 3-uniform hypergraphs}
\author{Siwei Lin, Qinghou Zeng\footnote{Research partially supported by National Key R\&D Program of China (Grant No. 2023YFA1010202) and National Natural Science Foundation of China (Grant No. 12371342). Email: zengqh@fzu.edu.cn (Corresponding author)}\\
{\small Center for Discrete Mathematics, Fuzhou University, Fujian, 350003, China}}
\date{}
\maketitle

\begin{abstract}
Let $k\ge2$ be a fixed integer. Bollob\'as and Scott proved that every
$3$-uniform hypergraph with $m$ edges admits a partition of its vertex set
into $k$ parts such that each part spans at most
$m/k^3+O_k(m^{6/7})$ 
edges. Scott later suggested that the error term should be
$O_k(m^{2/3})$. In this paper, we show that every
$3$-uniform hypergraph with $m$ edges admits a partition into $k$ parts
such that each part spans at most
$m/k^3+O_k(m^{2/3})$ 
edges, thereby establishing the proposed error term.
\end{abstract}



\section{Introduction}
Vertex-partition problems in graphs and hypergraphs ask for a decomposition
of the vertex set into disjoint classes satisfying prescribed combinatorial
conditions.  A classical example is the \emph{Max-Cut problem}: given a graph
$G$, find a bipartition $V(G)=V_1\cup V_2$ maximizing
$e(V_1,V_2)$.  The problem is NP-hard even for triangle-free cubic graphs
\cite{Yannakakis78}, and has been studied extensively in both combinatorics
and theoretical computer science.

It is easy to prove that every graph $G$ with $m$ edges has
a bipartition $V(G)=V_1\cup V_2$ satisfying $e(V_1,V_2)\ge m/2$.
Edwards \cite{Edwards73,Edwards75} proved the essentially best possible
improvement $e(V_1,V_2)\ge m/2+h(m)/4$, where $h(m)=\sqrt{2m+1/4}-1/2$.
Bollob\'as and Scott \cite{BS2002} extended this result to $k$-partitions.  They proved that every graph with $m$ edges admits a
partition $V(G)=V_1\cup\cdots\cup V_k$ such that $e(V_1,\ldots,V_k) \ge \frac{k-1}{k}(m+h(m)/2)-O(k)$.

Maximizing $e(V_1,\ldots,V_k)$ is equivalent to minimizing the total number
of edges lying in some $V_i$.  This does not, however,
control how the internal edges are distributed among the individual
classes.  So naturally we ask what happens if we want to minimize $\max_{i\in[k]} e(V_i)$.
For $k=2$, this is the \emph{bottleneck bipartition problem}, which was posed by
Entringer (see \cite{SS94}).  Porter \cite{Porter92} proved that every graph
$G$ with $m$ edges has a bipartition $V(G)=V_1\cup V_2$ such that $ e(V_i)\le m/4+O(\sqrt m)$ for $i=1,2$. 
Bollob\'as and Scott \cite{BS99} sharpened this bound to $e(V_i)\le m/4+h(m)/8$ for $i=1,2$.

More generally, in a uniformly random partition of the vertex set of a
graph into $k$ classes, a fixed edge lies inside a specified class with
probability $1/k^2$.  Thus $m/k^2$ is the natural leading term for the
number of edges in each class.  This leading term is asymptotically best
possible, as is seen by considering complete graphs.  Bollob\'as and Scott
\cite{BS99} proved that every graph with $m$ edges admits a partition
$V(G)=V_1\cup\cdots\cup V_k$ such that for every $i\in[k]$, $e(V_i) \le m/k^2+\frac{k-1}{2k^2}h(m)$.
The related problems have attracted many researchers with a wealth
of results \cite{ABKS03,BS93,BS99,BS00,BS2002,BS10,FHZ14,HWY16,HZ17,HX27,JMYY19,LLS16}.

We consider the analogous problem for uniform hypergraphs.  Let $H$ be an
$r$-uniform hypergraph with $m$ edges and write
\begin{align*}
    g_k(H)
    =
    \min_{V(H)=V_1\cup\ldots\cup V_k}
    \max_i\{e_H(V_i)\}\,.
\end{align*}
In a random vertex partition of $H$ into $k$ sets we expect to have $m/k^r$
edges in each set. By considering large complete $r$-uniform hypergraphs, it is easy
to see that we cannot demand that every set contains less than $m/k^r$ edges.  For graphs, that is, for $r=2$, the results above show that this bound can be achieved to a lower-order error term.
For $r\ge3$, Bollob\'as and Scott \cite{BS97} posed the following
conjecture.

\begin{conjecture}[Bollob\'as and Scott \cite{BS97}]
Let $r\ge3$ and $k\ge2$ be fixed integers.  Then every $r$-uniform
hypergraph $H$ with $m$ edges satisfies
\[
    g_k(H)\le \frac{m}{k^r}+o_k(m).
\]
\end{conjecture}

In the same paper, Bollob\'as and Scott verified the conjecture for
$3$-uniform hypergraphs, proving an error term of order $m^{6/7}$.
\begin{theorem}[Bollob\'as and Scott \cite{BS97}]
    Let $k\ge2$ be a fixed integer.  Then every 3-uniform hypergraph $H$ with $m$ edges satisfies
    \[
        g_k(H)\le \frac{m}{k^3}+O_k(m^{6/7}).
    \]
\end{theorem}

Scott \cite{Scott05} suggested that, in the $3$-uniform case, the error
term may be improved to order $m^{2/3}$.
\begin{conjecture}[Scott \cite{Scott05}]
    Let $k\ge2$ be a fixed integer.  Then every 3-uniform hypergraph $H$ with $m$ edges satisfies
    \[
        g_k(H)\le \frac{m}{k^3}+O_k(m^{2/3}).
    \]
\end{conjecture}
Indeed, let \(H\sim G^{(3)}(n,p)\), where \(p=1/2\). Standard concentration estimates show that with high
probability $ g_k(H)   \le  p\binom{\lceil n/k\rceil}{3}   +  O_{k}\!\left(
        n^{3/2}\sqrt{p\log n}    \right).$
Note that, with high probability, \(m=(1+o(1))p\binom {n}{3}\). This gives, with high probability, $g_k(H)\le m/k^3+O_k(\sqrt{m\log m})$, which means that the error term $O(m^{6/7})$ may be much improved. On the other hand, the maximum vertex degree of \(H\) is, with
high probability, $\Theta(m^{2/3})$.  Note that an individual high-degree vertex may create an error comparable to its degree.  Thus,
this suggests \(m^{2/3}\) as a target.

In this paper, we confirm Scott's conjecture.

\begin{theorem}\label{thm:main}
Let $k\ge2$ be a fixed integer.  Then every $3$-uniform hypergraph $H$ with $m$ edges satisfies
\[
        g_k(H)\le \frac{m}{k^3}+O_k(m^{2/3}).
\]
\end{theorem}

\section{Preliminaries}
For hypergraphs, we follow the notations in \cite{BM08}. Let \(r\) be an integer. An \(r\)-uniform hypergraph \(H\) is a pair \((V, E)\), where \(V\) is a finite set of vertices and \(E \subseteq \binom{V}{r}\) is a family of \(r\)-element subsets of \(V\) called edges. 
For a vertex \(v \in V\), the degree \(d_H(v)\) of \(v\) is the number of edges of \(H\) containing \(v\). 
For a pair of vertices $u,v\in V$, the codegree $d_H(u,v)$ of $u,v$ is the number of edges of $H$ containing $u,v$.
For a subset $U\subseteq V(H)$, let $(d_H)_2(U)=\sum_{uv\in\binom U2}d_H(u,v)$. Let $\Delta_H(U)=\max_{v\in U}d_H(v)$, and let $\Delta_{2,H}(U)=\max_{\{u,v\}\in \binom{U}{2}}d_H(u,v)$. 
Set $\Delta_H(U)=0$ if $U=\varnothing$ and $\Delta_{2,H}(U)=0$ if $|U|\le1$. Write $\Delta(H)=\Delta_H(V(H))$, $\Delta_2(H)=\Delta_{2,H}(V(H))$.

Let \(S, T\) be two disjoint subsets of \(V\). We write \(e_H(S) = |\{e \in E(H): e \subseteq S\}|\)
. When understood, the reference to \(H\) in the subscript may be dropped. Furthermore, for $S\subseteq V(H)$, we let $H[S]$ be the induced subhypergraph of $H$ on $S$, and let $H|_S$ be the restriction of $H$ to $S$ with multiplicities. That is, $H[S]$ has the edge set
\begin{align*}
    E(S):=E(H[S])=\{e\in E(H):e\subseteq S\},
\end{align*}
and $H|_S$ has the edge set
\begin{align*}
    E(H|_S)=\{e\cap S:e\in E(H),e\cap S\ne \varnothing\},
\end{align*}
where the right-hand side is regarded as a multiset.

Write $[k]=\{1,\dots,k\}$ and $S_k=\left\{(p_1,\dots,p_k)\in[0,1]^k: \sum_{i=1}^kp_i=1\right\}$.
Let $A$ be a finite set and $\alpha:A\to \mathbb R_{\ge0}$ and $\beta:\binom {A}{2}\to \mathbb R_{\ge0}$ be two nonnegative functions. For disjoint $S,T\subseteq A$, put
\begin{align*}
    \alpha(S)=\sum_{a\in S}\alpha_a,\quad
    \beta(S)=\sum_{ab\in\binom S2}\beta_{ab},\quad
    \text{and}\quad
    \beta(S,T)=\sum_{a\in S,\,b\in T}\beta_{ab}.
\end{align*}

A \emph{bad-set configuration} is a tuple
        $\cS=(A;m_0,\alpha,\beta,\Gamma),$
where $m_0\ge0$, $\alpha:A\to\mathbb R_{\ge0}$ is a nonnegative vertex weighting on $A$, $\beta:\binom{A}{2}\to\mathbb R_{\ge0}$ is a nonnegative weighted graph on $A$, and $\Gamma$ is a simple $3$-graph on $A$. We write
\[
        M(\cS)=m_0+\alpha(A)+\beta(A)+e(\Gamma).
\]

A \textit{fractional $k$-coloring} of $A$ is a family
$x=(x_{a,i})_{a\in A,i\in[k]}$ such that
$x_a=(x_{a,1},\ldots,x_{a,k})\in S_k$ for every $a\in A$.
It is called \textit{integral} if $x_{a,i}\in\{0,1\}$ for all $a\in A$ and
$i\in[k]$. For $p=(p_1,\ldots,p_k)\in S_k$ and $i\in[k]$, define
\begin{align*}
 L_i(\cS;p,x)
 =p_i^3m_0+p_i^2\sum_{a\in A}\alpha_ax_{a,i}
   +p_i\sum_{ab\in\binom A2}\beta_{ab}x_{a,i}x_{b,i},
   \qquad
 G_i(\Gamma;x)
 =\sum_{abc\in E(\Gamma)}x_{a,i}x_{b,i}x_{c,i}.
 \end{align*}
 Let  $K_i(\cS;p,x)
 =L_i(\cS;p,x)+G_i(\Gamma;x).$ 
If $A=A_1\cup\cdots\cup A_k$ is an integral $k$-coloring of $A$, for $i\in[k]$, we also write 
\[
 K_i(\cS;p,A_i)
 =p_i^3m_0+p_i^2\alpha(A_i)+p_i\beta(A_i)+e_\Gamma(A_i).
\]

\section{Hypergraphs with small maximum codegree}

Let $H$ be a 3-graph and let $V(H)=A\cup U$ be a bipartition of the vertex set of $H$. We call the vertices of $A$ \emph{bad} and $A$ the \emph{bad set}.  And we call the vertices of $U$ \emph{good} and $U$ the \emph{good set}. 
As observed in \cite{Scott05}, the main difficulty comes from vertices of large degree. When all degrees are sufficiently small, a random partition together with a martingale inequality already gives useful concentration. A standard strategy is therefore to separate a small set of high-degree vertices from the vertex set as bad vertices and to label the remaining vertices good. We follow this method, but during the proof some good vertices will also be promoted to the bad set in order to control pair-codegrees.

In this section, we consider a simpler case: a 3-uniform hypergraph in which the codegree of every pair of good vertices is relatively small.

\subsection{Partitioning most of the bad vertices}
\label{subsec:partition A}
We start with the partition of $H|_A$. Indeed, our subsequent result concerns general mixed 3-multigraphs where every 3-edge has multiplicity 1.

Let $V$ be a set of size $n$, $w:\binom V2\to\mathbb R_{\ge0}$ be a
nonnegative weighted graph, and $F$ be a simple $3$-graph on $V$.
For every $i\in[k]$, let $\gamma_i\ge0$, $c_i\in\mathbb R$, and $\ell_{u,i}\in\mathbb R$ for $u\in V$. 
For a fractional $k$-coloring $x$ of $V$, define
\begin{align*}
  P_i(x)=c_i+\sum_{u\in V}\ell_{u,i}x_{u,i}
       +\gamma_i\sum_{uv\in\binom V2}w_{uv}x_{u,i}x_{v,i}
       +\sum_{uvw\in E(F)}x_{u,i}x_{v,i}x_{w,i}.
\end{align*}

\begin{lemma}\label{lem:graph-preserving}
Let $V$ be a set of size $n$, $w:\binom V2\to\mathbb R_{\ge0}$ be a
nonnegative weighted graph, and $F$ be a simple $3$-graph on $V$.
Let $P_i(x)$ be the functions defined above.
For every fractional $k$-coloring $x$ of $V$, there is a fractional
$k$-coloring $x'$ such that
\begin{enumerate}[label=\textup{(\roman*)},leftmargin=2.3em]
\item $\{i\in[k]:x'_{v,i}\ne0\}\subseteq\{i\in[k]:x_{v,i}\ne0\}$ for every $v\in V$;
\item at most $k(k-1)$ vertices are non-integral;
\item for every $i\in[k]$,
\begin{align*}
    P_i(x')\le P_i(x)+12\sqrt6(k-1)\sqrt{e(F)\Delta_2(F)}.
\end{align*}
\end{enumerate}
\end{lemma}

\begin{proof}[\bf Proof.]
For colors $a<b$, let
\[
        R_{ab}=\{v\in V:x_{v,a}>0\text{ and }x_{v,b}>0\}.
\]
While some $R_{ab}$ has size $r\ge3$, choose a triple
$Q\in\binom{R_{ab}}3$ minimizing $(d_F)_2(Q)$.
By the exact multilinear expansion, for any displacement $\delta$ and $i\in[k]$ we have
\begin{align}\label{eq:Taylor-expansion-of-P(x)}
    P_i(x+\delta)-P_i(x)
    =\sum_{v\in Q}\frac{\partial P_i}{\partial x_{v,i}}\delta_{v,i}
    +\frac12\sum_{v\in Q}\sum_{u\ne v} \left(\gamma_iw_{uv}+\sum_{w:uvw\in E(F)}x_{w,i}\right)\delta_{u,i}\delta_{v,i}\notag\\
    +\sum_{uvw\in E(F)}\delta_{u,i}\delta_{v,i}\delta_{w,i}.
\end{align}

Let $B_Q$ be the symmetric matrix indexed by $Q$ with zero diagonal and
$(B_Q)_{uv}=w_{uv}/2$ off the diagonal.

\begin{claim}\label{clm:inertia}
Let $B$ be a real symmetric $3\times3$ matrix with zero diagonal and
nonnegative off-diagonal entries.  For every $g\in\R^3$, there is a nonzero
$h\in\R^3$ such that $g\cdot h=0$ and $h^{\top}Bh\le0$.
\end{claim}
\begin{proof}
If the off-diagonal entries are $a,b,c\ge0$, then
\[
        \det(\lambda I-B)
        =\lambda^3-(a^2+b^2+c^2)\lambda-2abc.
\]
This polynomial has at most one positive root, so the positive index of
inertia of $B$ is at most one.  Since $g^\perp$ has dimension at least two,
the quadratic form of $B$ cannot be positive definite on $g^\perp$.
Therefore there is a nonzero $h\in g^\perp$ with
$h^{\top}Bh\le0$.
\end{proof}

Apply the claim to $g=\left(\frac{\partial P_a}{\partial x_{v,a}}\right)_{v\in Q}$ and $B=B_Q$.  We obtain a nonzero $h\in\R^Q$ satisfying
$g\cdot h=0$ and $h^{\top}B_Qh\le0$.  Replace $h$ by $-h$ if necessary
so that
 $-\sum_{v\in Q}\frac{\partial P_b}{\partial x_{v,b}}h_v\le0.$
Move along the direction
\begin{align*}
 x_{v,a}(t)=
 \begin{cases}
     x_{v,a}+th_v,&v\in Q;\\
     x_{v,a},&v\notin Q,
 \end{cases}
 \qquad
 x_{v,b}(t)=
 \begin{cases}
     x_{v,b}-th_v,&v\in Q;\\
     x_{v,b},&v\notin Q,     
 \end{cases}   
\end{align*}
until the first boundary point of the product of simplices is reached.  
Since both displayed coordinates were positive at every $v\in Q$, at least one coordinate at some $v\in Q$ becomes zero.

Let $\delta=(\delta_{v,i})_{v\in V,i\in[k]}$ be the family such that $\delta_{v,a}=th_v,\delta_{v,b}=-th_v$ for $v\in Q$ and $\delta=0$ in all other cases. 
Then by the definition of $h$, for $i=a,b$ we have 
\begin{align*}
    \sum_{v\in V}\frac{\partial P_i}{\partial x_{v,i}}\delta_{v,i}\le0,
    \quad
    \sum_{v\in Q}\sum_{u\ne v} w_{uv}\delta_{u,i}\delta_{v,i}\le0.
\end{align*}
By \eqref{eq:Taylor-expansion-of-P(x)}, we have
\begin{align*}
    P_a(x+\delta)-P_a(x)
    &\le\frac12\sum_{v\in Q}\sum_{u\ne v} \sum_{w:uvw\in E(F)}x_{w,a}\delta_{u,a}\delta_{v,a}+\sum_{uvw\in E(F)}\delta_{u,a}\delta_{v,a}\delta_{w,a}
    \le 2(d_{F})_2(Q),\\
    P_b(x+\delta)-P_b(x)
    &\le\frac12\sum_{v\in Q}\sum_{u\ne v} \sum_{w:uvw\in E(F)}x_{w,b}\delta_{u,b}\delta_{v,b}+\sum_{uvw\in E(F)}\delta_{u,b}\delta_{v,b}\delta_{w,b}
    \le 2(d_{F})_2(Q).
\end{align*}

Averaging over all triples of $R_{ab}$ together with the minimality of $(d_F)_2(Q)$ gives
\[
 (d_F)_2(Q)
 \le \frac{6}{r(r-1)}
      \sum_{uv\in\binom{R_{ab}}2}d_F(u,v)
 \le \min\left\{\frac{18e(F)}{r(r-1)},3\Delta_2(F)\right\}.
\]
Every operation using the color pair $\{a,b\}$ strictly decreases
$|R_{ab}|$.  If $e(F)\Delta_2(F)=0$, then this operation cannot increase any $P_i(x)$.  Otherwise, set $R_0=\sqrt{6e(F)/\Delta_2(F)}\ge1$ and we have
\begin{align*}
    \sum_{r=3}^n\min\left\{\frac{18e(F)}{r(r-1)},3\Delta_2(F)\right\}
    &\le \sum_{r\le R_0}3\Delta_2(F)+\sum_{r> R_0}\frac{18e(F)}{r(r-1)}\\
    &\le 3(\lfloor R_0\rfloor-2) \Delta_2(F)+\frac{18e(F)}{\lfloor R_0\rfloor}\\
    &\le6\sqrt6\sqrt{e(F)\Delta_2(F)}.
\end{align*}
It means that all operations using the color pair $\{a,b\}$ increase $P_a(x),\,P_b(x)$ by at most $12\sqrt6\sqrt{e(F)\Delta_2(F)}$.
Let $x'$ be the final fractional coloring. Then 
\begin{align*}
    P_i(x')-P_i(x)\le 12\sqrt6(k-1)\sqrt{e(F)\Delta_2(F)}.
\end{align*}

At termination, there are at most two vertices $u,v$ for each color pair $i,j$ such that $x'_{w,k}\notin\{0,1\}$ for $w\in\{u,v\},\,k\in\{i,j\}$. Hence at most $k(k-1)$ vertices remain non-integral.
\end{proof}

\subsection{Partitioning remaining bad vertices}
\label{subsec: partition A and give p for U}
In this subsection, we partition the remaining $O(1)$ bad vertices. We aim to give an integral $k$-coloring of $A$ such that the increase of $G_i(H;\xi)$ only relies on $H[A]$.

Fix $\mu,T\ge0$.  For $a,c\ge0$, define 
\[
 \rho_{\mu,T}(a,c)=\sup\{r\ge0:\mu r^3+ar^2+cr\le T\}.
\]
Note that if $\mu+a+c>0$ then $\rho_{\mu,T}(a,c)$ is exactly the unique solution of $\mu r^3+ar^2+cr= T$, and if $\mu=a=c=0$ then $\rho_{\mu,T}(a,c)$ takes the value $+\infty$.

\begin{lemma}\label{lem:capacity-convex}
Let $\mu,T,a,c,u,v\ge0$.  The function
$f(t)=\rho_{\mu,T}(a+ut,c+vt)$
is convex, in the extended-valued sense, on every interval $I$ on which
$a+ut\ge0$ and $c+vt\ge0$.
\end{lemma}

\begin{proof}[\bf Proof.]
Suppose first that $\mu+(a+ut)+(c+vt)>0$ on $I$ and $T>0$, i.e., $f(t)$ is finite and positive.  Write
$A=a+ut$, $C=c+vt$, and $r=\rho_{\mu,T}(A,C)$.  Then
\[
        \mu r^3+Ar^2+Cr=T,
        \qquad
        D=3\mu r^2+2Ar+C>0.
\]
Differentiating twice and simplifying yields
\begin{equation}
 r''=\frac{2r(ur+v)}{D^3}
 \bigl(3uAr^2+vAr+2uCr+vC+3u\mu r^3\bigr)\ge0.
 \label{eq:capacity-second-derivative}
\end{equation}
Thus the assertion holds whenever the root is finite and positive.

If $T=0$, then $r\equiv0$; the only other alternatives are the
constant function $r\equiv+\infty$ or  a function that is $+\infty$ at an endpoint and 0 elsewhere. These functions are convex in the extended-valued sense.  Assume now that
$T>0$.  For $\eps>0$, replace the two coefficients $A$ and $C$ by
$A+\eps$ and $C+\eps$.  Then $\rho_{\mu,T}(A+\eps,C+\eps)$ is finite and positive.
Hence $\rho_{\mu,T}(A+\eps,C+\eps)$ is convex by \eqref{eq:capacity-second-derivative}, and converges pointwise
to $\rho_{\mu,T}(A,C)$ as $\eps\downarrow0$. This proves the remaining cases.
\end{proof}

\begin{lemma}\label{lem:lower-rank-completion}
Let $\mathcal S=(A;m_0,\alpha,\beta,\Gamma)$ be a bad-set
configuration. 
Let $T_1,\ldots,T_k\ge0$.  Suppose that, for some $p^0\in S_k$, a fractional $k$-coloring $x$ of $A$, and $i\in[k]$,
        $L_i(\mathcal S;p^0,x)\le T_i.$
Then there are an integral coloring $z$ and a vector $p\in S_k$
such that, for every $i\in[k]$,
\[
        L_i(\mathcal S;p,z)\le T_i,
\]
and every vertex that is integral in $x$ may be kept fixed in $z$.
\end{lemma}

\begin{proof}[\bf Proof.]
For a current fractional coloring $x$, for $i\in[k]$, put
$a_i=\sum_{v\in A}\alpha_vx_{v,i}$ and $c_i=\sum_{uv\in\binom A2}\beta_{uv}x_{u,i}x_{v,i}$. Define
\[
        \Psi(x)=\sum_{i=1}^k\rho_{m_0,T_i}(a_i,c_i).
\]
We first state the following result.
\begin{claim}
    For fixed nonnegative $a_i,c_i,T_i$, there exists $p\in S_k$ with
\[
        m_0p_i^3+a_ip_i^2+c_ip_i\le T_i
\]
for $i\in[k]$ if and only if
\[
        \sum_{i=1}^k\rho_{m_0,T_i}(a_i,c_i)\ge1.
\]
\end{claim}
\begin{proof}
    Necessity is immediate.  Conversely, choose finite numbers $r_i\in[0,\rho_{m_0,T_i}(a_i,c_i)]$ such that $\sum_{i=1}^kr_i\ge1$. Set $R=\sum_i r_i$.
If the displayed sum is at least $1$, i.e., $R\ge1$, then choose $p\in S_k$ by setting $p_i=r_i/R$. Then $p$ satisfies $p_i\le\rho_{m_0,T_i}(a_i,c_i)$ for every $i$.  This proves the claim.
\end{proof}

This Claim shows that $\Psi(x)\ge1$ and it suffices to find an integral coloring $z$ such that $\Psi(z)\ge1$. 

Now fix a non-integral vertex $v$ and keep all other coordinates fixed.  For each
$i\in[k]$ there are nonnegative numbers $a_{i,0},c_{i,0},c_{i,1}$ such that $a_i=a_{i,0}+\alpha_vt,\,c_i=c_{i,0}+c_{i,1}t$ where $t=x_{v,i}$.
Hence
\[
 f_i(t)=\rho_{m_0,T_i}(a_{i,0}+\alpha_vt,c_{i,0}+c_{i,1}t)
\]
is convex on $[0,1]$ by Lemma~\ref{lem:capacity-convex}.  Let $e_j$ denote
the $j$th unit vector of $S_k$.  The convexity gives $f_i(x_{v,i})\le (1-x_{v,i})f_i(0)+x_{v,i}f_i(1)$ for any $v$ and $i$. Then we have
\begin{align*}
 \Psi(x)
 &
 \le
 \sum_{i=1}^k\left(\sum_{j\ne i}x_{v,j}f_i(0)+x_{v,i}f_i(1)\right)
 =\sum_{j=1}^k x_{v,j}\,\Psi(x|_{x_v=e_j})
 \le\max_{j\in[k]}\Psi(x|_{x_v=e_j}).
\end{align*}
Thus $v$ can be assigned one color without decreasing $\Psi$.  Repeating
this operation rounds every non-integral vertex while preserving
$\Psi\ge1$.  Let $z$ be the final integral coloring. Then we have $\Psi(z)\ge1$. This completes the proof.
\end{proof}

\subsection{Partitioning good vertices}
\label{subsec: partition U}
In this subsection, we partition all good vertices. Our next result is weaker than Lemma~\ref{lem:graph-preserving}, but more general in scope.

Let $\mathcal B$ be a finite set of blocks and $\Sigma$ be a finite state
set.  A fractional state assignment is a family
$y=(y_{B,\sigma})$ such that $y_{B,\sigma}\ge0$ for every $B,\sigma$, and $\sum_{\sigma\in\Sigma}y_{B,\sigma}=1$ for every $B$.
For every $i\in[k]$, let 
\begin{align*}
    \bar P_i(y)
    =C_i
    +\sum_{B\in \mathcal B}\sum_{\sigma\in \Sigma}C^{(i)}_{\{(B,\sigma)\}}y_{B,\sigma}
    +\sum_{\{B,C\}\in \binom{\mathcal B}{2}}\sum_{\sigma,\tau\in \Sigma}C^{(i)}_{\{(B,\sigma),(C,\tau)\}}y_{B,\sigma}y_{C,\tau}\\
    +\sum_{\{B,C,D\}\in \binom{\mathcal B}{3}}\sum_{\sigma,\tau,\rho\in \Sigma}C^{(i)}_{\{(B,\sigma),(C,\tau),(D,\rho)\}}y_{B,\sigma}y_{C,\tau}y_{D,\rho},
\end{align*}
where all the coefficients are reals. 
For distinct $B,C\in \mathcal B$, let $w_{BC}$ be the sum of $\left|C^{(i)}_{\{(B_1,\sigma_1),\ldots,(B_r,\sigma_r)\}}\right|$ where $B,C\in \{B_1,\ldots,B_r\}$ and $\sigma_1,\ldots,\sigma_{r}\in\Sigma$, that is, 
\begin{align*}
    w_{BC}
    =\sum_{i=1}^k\sum_{\sigma,\tau\in \Sigma}\left|C^{(i)}_{\{(B,\sigma),(C,\tau)\}}\right|
    +\sum_{i=1}^k\sum_{D\ne B,C}\sum_{\sigma,\tau,\rho\in \Sigma}\left|C^{(i)}_{\{(B,\sigma),(C,\tau),(D,\rho)\}}\right|.
\end{align*}
Furthermore, define
    $W=\sum_{BC\in\binom{\mathcal B}{2}}w_{BC},
    \,
    \Omega=\max_{BC\in\binom{\mathcal B}{2}}w_{BC}.$
Here we set $\Omega=0$ if $|\mathcal B|\le1$. Define
\begin{align*}
    D_B
    =\sum_{i=1}^k\sum_{\sigma\in \Sigma}\left|C^{(i)}_{\{(B,\sigma)\}}\right|
    +\sum_{i=1}^k\sum_{C\ne B}\sum_{\sigma,\tau\in \Sigma}\left|C^{(i)}_{\{(B,\sigma),(C,\tau)\}}\right|+\sum_{i=1}^k\sum_{\{C, D\}\in \binom{\mathcal B\setminus\{B\}}{2}}\sum_{\sigma,\tau,\rho\in \Sigma}\left|C^{(i)}_{\{(B,\sigma),(C,\tau),(D,\rho)\}}\right|,
\end{align*}
and write $D=\max_{B\in \mathcal B}D_B$.

\begin{lemma}\label{lem:signed-rounding}
Let $\mathcal B$ be a finite set of blocks, $\Sigma$ be a finite state
set of size $q$, and $y$ be a fractional state assignment.
Let $\bar P_1(y),\ldots,\bar P_k(y)$ be polynomials defined as above.
Then for any fractional state assignment $y$, there is an integral state assignment $z$ such that, for $i\in[k]$,
\[
        \bar P_i(z)\le \bar P_i(y)
        +2k(k+1)q(q-1)\left(\sqrt{W\Omega}+D\right).
\]
\end{lemma}
\begin{proof}[\bf Proof.]
The case $q=1$ is trivial, so we assume that $q\ge2$. Fix two distinct states $\sigma,\tau\in\Sigma$ and let
\[
 \mathcal R_{\sigma\tau}
 =\{B\in\mathcal B:y_{B,\sigma}>0\text{ and }y_{B,\tau}>0\}.
\]
While $r=|\mathcal R_{\sigma\tau}|\ge k+1$, choose
$\mathcal Q\in\binom{\mathcal R_{\sigma\tau}}{k+1}$ minimizing
$ w(\mathcal Q)=\sum_{BC\in \binom{\mathcal Q}{2}}w_{BC}$.
For every $i\in[k]$, form the vector $g_i\in\R^Q$ with coordinates
 $(g_i)_B=\frac{\partial\bar P_i}{\partial y_{B,\sigma}}
          -\frac{\partial\bar P_i}{\partial y_{B,\tau}}.$
The $k$ vectors $g_i$ lie in a $(k+1)$-dimensional space, so there is a
nonzero $h\in\R^{\mathcal Q}$ orthogonal to all of them.  Move along
\[
 y_{B,\sigma}(t)=
 \begin{cases}
    y_{B,\sigma}+th_B,&B\in \mathcal Q;\\
    y_{B,\sigma},&B\notin \mathcal Q,
 \end{cases}
 \qquad
 y_{B,\tau}(t)=
 \begin{cases}
    y_{B,\tau}-th_B,&B\in \mathcal Q;\\
    y_{B,\tau},&B\notin \mathcal Q
 \end{cases}
\]
to either nonzero endpoint of the maximal feasible interval.  Every
first-order change vanishes, and at least one nonzero coordinate becomes zero.
Let $\delta=(\delta_{B,\rho})_{B\in \mathcal B,\rho\in\Sigma}$ be the family such that 
\begin{align*}
\delta_{B,\rho}=
\begin{cases}
    th_B,&B\in \mathcal Q,\rho=\sigma;\\
    -th_B,&B\in \mathcal Q,\rho=\tau;\\
    0,&B\notin \mathcal Q\ \text{or}\ \rho\ne\sigma,\tau.
\end{cases}
\end{align*}
Since $\nabla\bar P_i(y)\cdot\delta\le0$, $|\delta_{B,\sigma}|\le1$ and $0\le y_{B,\sigma}\le1$ for every $B,\sigma$, we have that, for every $i\in[k]$,
\begin{align*}
    \Delta\bar P_i(y)&=\sum_{B\in \mathcal Q}\sum_{\rho\in\Sigma}\frac{\partial \bar P_i}{\partial y_{B,\rho}}\delta_{B,\rho}
    +\sum_{\{B,C\}\in\binom{\mathcal Q}{2}}\sum_{\rho_1,\rho_2\in\{\sigma,\tau\}}\frac{\partial^2 \bar P_i}{\partial y_{B,\rho_1}\partial y_{C,\rho_2}}\delta_{B,\rho_1}\delta_{C,\rho_2}\\
    &\hspace{25mm}+\sum_{\{B,C,D\}\in \binom{\mathcal Q}{3}}\sum_{\rho_1,\rho_2,\rho_3\in \{\sigma,\tau\}}C^{(i)}_{(B,\rho_1),(C,\rho_2),(D,\rho_3)}\delta_{B,\rho_1}\delta_{C,\rho_2}\delta_{D,\rho_3}\\
    &\le\sum_{\{B,C\}\in\binom{\mathcal Q}{2}}\sum_{\rho_1,\rho_2\in\{\sigma,\tau\}}\left|\frac{\partial^2 \bar P_i}{\partial y_{B,\rho_1}\partial y_{C,\rho_2}}\right|
    +\sum_{\{B,C,D\}\in \binom{\mathcal Q}{3}}\sum_{\rho_1,\rho_2,\rho_3\in \{\sigma,\tau\}}\left|C^{(i)}_{(B,\rho_1),(C,\rho_2),(D,\rho_3)}\right|\\
    &\le2w(\mathcal Q),
\end{align*}
where $\Delta\bar P_i(y)=\bar P_i(y+\delta)-\bar P_i(y)$ and
\begin{align*}
    \frac{\partial^2 \bar P_i}{\partial y_{B,\rho_1}\partial y_{C,\rho_2}}=C^{(i)}_{(B,\rho_1),(C,\rho_2)}+\sum_{D\ne B,C}\sum_{\rho_3\in\Sigma}C^{(i)}_{(B,\rho_1),(C,\rho_2),(D,\rho_3)}y_{D,\rho_3}.
\end{align*}
Averaging over the $(k+1)$-subsets of $R_{\sigma\tau}$ gives
\[
 w(\mathcal Q)
 \le \frac{k(k+1)}{r(r-1)}
      \sum_{BC\in\binom {R_{\sigma\tau}}{2}}w_{BC}
 \le k(k+1)\min\left\{\frac W{r(r-1)},\Omega\right\}.
\]
Every operation for the state pair $\{\sigma,\tau\}$ strictly decreases
$|\mathcal R_{\sigma\tau}|$, and supports never increase.  If $W\Omega>0$, splitting
at $\sqrt{W/\Omega}$ gives
\[
 \sum_{r=k+1}^{|\mathcal B|}\min\left\{\frac W{r(r-1)},\Omega\right\}
 \le2\sqrt{W\Omega},
\]
while when $W\Omega=0$ the sum is zero. Let $y'$ be the final fractional state assignment. Summing over $\binom {q}2$ state pairs gives, for every $i\in[k]$,
\begin{align*}
    \bar P_i(y')-\bar P_i(y)\le 2k(k+1)q(q-1)\sqrt{W\Omega}.
\end{align*}

For any fixed state pair $\sigma,\tau$, there are at most $k$ blocks $B$ such that $y'_{B,\sigma}\ne0$ and $y'_{B,\tau}\ne0$. Moreover, for any non-integral block $B$, there are at least two states $\sigma,\tau$ such that $y'_{B,\sigma}\ne0$ and $y'_{B,\tau}\ne0$.
So $y'$ has at most $k\binom {q}2$ blocks remaining non-integral.  Round those blocks arbitrarily, one at a time. Let $z$ be the integral state assignment after rounding. Since rounding one block changes any $\bar P_i(y')$ by at most $D$, we have that, for every $i\in[k]$,
\begin{align*}
    \bar P_i(z)\le\bar P_i(y')+k\frac {q^2}{2}D\le\bar P_i(y)+2k(k+1)q(q-1)\left(\sqrt{W\Omega}+D\right),
\end{align*}
which proves the lemma.
\end{proof}

Let $H$ be a 3-graph on a vertex set $V$. Next we may round all good vertices with Lemma~\ref{lem:signed-rounding} instead of probabilistic techniques.

\begin{corollary}\label{cor:residual-rounding}
Let $H$ be a $3$-graph with $m$ edges and let
$V(H)=A\cup U$.  Let
$A=A_1\cup\cdots\cup A_k$ be a fixed
$k$-coloring of $A$, and let $p=(p_1,\ldots,p_k)\in S_k$. Let $\cS=(A;m_0,\alpha,\beta,\Gamma)$ be the bad-set configuration defined above.
  Then there is a $k$-partition
$V(H)=V_1\cup\cdots\cup V_k$ with
$A_i\subseteq V_i$ and
\[
 e_H(V_i)\le K_i(\cS;p,A_i)
 +2k^5\left(\sqrt{3m\Delta_{2,H}(U)}+\Delta_H(U)\right)
\]
for every $i$.
\end{corollary}
\begin{proof}[\bf Proof.]
For each $u\in U$, use the state set $\Sigma=[k]$ and variables
$y_{u,1},\ldots,y_{u,k}$ with $\sum_i y_{u,i}=1$.  For $a\in A$, regard
$y_{a,i}=1$ when $a\in A_i$ and $y_{a,i}=0$ otherwise. For every $i\in[k]$, define
\[
        \bar P_i(y)=\sum_{abc\in E(H)}y_{a,i}y_{b,i}y_{c,i},
\]
where only the coordinates indexed by $U$ are variables.  At the
assignment $y_{u,i}=p_i$ for all $u\in U$, classifying an edge by the number
of its vertices in $A$ gives $\bar P_i(y)=K_i(\cS;p,A_i)$.

Use the notation $w_{uv},W,\Omega,D$ from Lemma~\ref{lem:signed-rounding}
for $\bar P_1,\ldots,\bar P_k$. Note that an edge containing a
fixed pair $uv\in\binom U2$ contributes at most $k$ to $w_{uv}$, and an edge containing a fixed vertex $u\in U$ contributes at most $k$ to $D_u$.  Consequently
$W\le3km,\,\Omega\le k\Delta_{2,H}(U),\,D\le k\Delta_H(U).$
By Lemma~\ref{lem:signed-rounding} with $q=k$, there is an integral state
assignment $z$ such that, for every $i\in[k]$,
\[
 \bar P_i(z)\le K_i(\cS;p,A_i)
 +2k^5\left(\sqrt{3m\Delta_{2,H}(U)}+\Delta_H(U)\right).
\]
Set $V_i=A_i\cup\{u\in U:z_{u,i}=1\}.$ Then $\bar P_i(z)=e_H(V_i)$ for every $i$, which proves the corollary.
\end{proof}

 By combining with Lemma~\ref{lem:graph-preserving}, Lemma~\ref{lem:lower-rank-completion} and Corollary~\ref{cor:residual-rounding}, we can deal with the case for 3-graphs with small maximum codegree.

\begin{theorem}\label{thm:low codegree case}
Fix $k\ge2$.  Let $V(H)=A\cup U$ and let $H$ be a simple
$3$-graph with $m$ edges. 
  Then there is a $k$-partition
$V(H)=V_1\cup\cdots\cup V_k$ such that
\[
 \max_{i\in[k]}e_H(V_i)
 \le \frac m{k^3}
 +\left(\frac {k^2}{2}+12k\right)|A|^2+2k^5\left(\sqrt{3m\Delta_{2,H}(U)}+\Delta_H(U)\right).
\]
\end{theorem}
\begin{proof}[\bf Proof.]
Let $m_0=e_H(U)$ and $\Gamma=H[A]$.
For two distinct vertices $a,b\in A$, define
\begin{align*}
 \alpha_a=|\{\{a,u,v\}\in E(H):u,v\in U\}|,~\beta_{ab}=|\{\{a,b,u\}\in E(H):u\in U\}|.
\end{align*}
Let $\cS=(A;m_0,\alpha,\beta,\Gamma)$.  Then $M(\cS)=m$.
Set $p_i^0=1/k$ and $x_{a,i}=1/k$ for every $a\in A$ and $i\in[k]$.
Put $\Phi_i(x)=K_i(\cS;p^0,x).$ Then, for every $i\in[k]$,
\[
        \Phi_i(x)=\frac{M(\cS)}{k^3}=\frac m{k^3}.
\]
Moreover,
\[
 \Phi_i(x)=\frac{m_0}{k^3}
 +\frac1{k^2}\sum_{a\in A}\alpha_ax_{a,i}
 +\frac1k\sum_{ab\in\binom A2}\beta_{ab}x_{a,i}x_{b,i}
 +\sum_{abc\in E(\Gamma)}x_{a,i}x_{b,i}x_{c,i}.
\]
Thus applying Lemma~\ref{lem:graph-preserving} with the common weighted graph
$\beta$, $\gamma_i=1/k$ and $F=\Gamma$ gives a fractional coloring
$x'$ with at most $k(k-1)$ non-integral vertices such that, for every $i\in[k]$,
\begin{equation}\label{eq:finite-partial-m1}
        K_i(\cS;p^0,x')=\Phi_i(x')\le\frac m{k^3}+12(k-1)|A|^2,
\end{equation}
where the inequality holds by $e(F)\le\binom{|A|}{3}$ and $\Delta_2(F)\le|A|$.
For every $i\in[k]$, put
\[
        T_i=\frac m{k^3}+12(k-1)|A|^2-G_i(\Gamma;x').
\]
By \eqref{eq:finite-partial-m1}, we have $T_i\ge0$ and
$L_i(\cS;p^0,x')\le T_i$.
Lemma~\ref{lem:lower-rank-completion} therefore yields an integral
coloring $z$, which keeps every integral vertex of $x'$ fixed, and a
vector $p\in S_k$ such that, for every $i\in[k]$, $L_i(\cS;p,z)\le T_i$.
Note that only edges of the simple $3$-graph $\Gamma$
incident with one of the final $k(k-1)$ vertices can change their contribution to $G_i(\Gamma;x')$, and the number of such edges is at most
\[
        k(k-1)\binom{|A|-1}{2}\le\frac{k^2}{2}|A|^2.
\]
Let $A_i=\{a\in A:z_{a,i}=1\}$. Then, for every $i\in[k]$,
\[
        K_i(\cS;p,A_i)=L_i(\cS;p,z)+G_i(\Gamma;z)\le\frac m{k^3}+\left(\frac {k^2}{2}+12k\right)|A|^2.
\]
Applying Corollary~\ref{cor:residual-rounding} to the fixed coloring
$A=A_1\cup\cdots\cup A_k$ gives a $k$-partition of $V(H)$ with
\[
 e_H(V_i)\le\frac m{k^3}
 +\left(\frac {k^2}{2}+12k\right)|A|^2
 +2k^5\left(\sqrt{3m\Delta_{2,H}(U)}+\Delta_H(U)\right)
\]
for $i\in[k]$, as required.
\end{proof}

\section{Hypergraphs with large maximum codegree}

In this section, we consider the case for general 3-graphs. 

If there are many distinct pairs of vertices with large codegree in a 3-uniform hypergraph, we can force the two vertices in each pair to lie in different parts. This ensures that no edge containing such a pair is contained entirely in a single part, and hence we may delete all such edges.

Let $A$ and $P$ be disjoint sets, and let $M=\{u_iv_i:1\le i\le t\}$ be a matching whose endpoint set is $P$.  Write $B_i=\{u_i,v_i\}$. 
We construct a bad-set configuration $\cS=(A\cup P;m_0,\alpha,\beta,\Gamma)$, where $m_0$ is a constant, $\alpha,\beta$ are two nonnegative functions and $\Gamma$ is a simple 3-graph on $A\cup P$.  
For distinct $z,w\in A\cup P$, put 
\[
 D_1(z)=\alpha_z+\sum_{w\ne z}\beta_{zw}+d_\Gamma(z),
 \qquad
 D_2(zw)=\beta_{zw}+d_\Gamma(zw).
\]
For each $B_i$, we assume that $D_2(B_i)=0$, which means that, for $1\le i\le t$,
\begin{equation}
 \beta_{u_iv_i}=0
 \quad\text{and}\quad
 \{e\in E(\Gamma):B_i\subseteq e\}=\varnothing.
 \label{eq:matched-atoms-deleted}
\end{equation}
Thus for any subset $S$ of $A\cup P$ of size at most 3, the coefficient of $\prod_{a\in S}x_a$ is positive only if $|S\cap B_i|\le1$ for each $i$.

\begin{theorem}\label{thm:forced-interface}
Fix $k\ge2$.   Let
$N\ge1$ and $\varrho,\Upsilon,\Lambda\ge1$.  Suppose
$\cS=(A\cup P;m_0,\alpha,\beta,\Gamma)$ defined above satisfies $|A\cup P|\le\varrho N$ and for $z\in P$, $D_1(z)\le\Upsilon N^2$. For $z,w\in P$, suppose that 
$D_2(zw)\le\Lambda N$ if $z,w$ belong to distinct blocks, and $D_2(zw)=0$ if $z,w$ belong to the same block. Then there are a $k$-coloring $A\cup P=A_1\cup\cdots\cup A_k$ in which the vertices of every $B_i$ receive distinct colors, and a vector
$p\in S_k$, such that, for $i\in[k]$,
\[
 K_i(\mathcal S;p,A_i)\le \frac{M(\cS)}{k^3}
 +(12k+k^2)\varrho^2N^2 +4k^{10}\left(\sqrt{3M(\cS)\Lambda N}+\Upsilon N^2\right).
\]
\end{theorem}

\begin{proof}[\bf Proof.]
Let $\Sigma=\{(j,l)\in[k]^2:j\ne l\}$, and $q=|\Sigma|=k(k-1)$.
For a block $B_r=\{u_r,v_r\}$, the state $(j,l)$ assigns color $j$ to
$u_r$ and color $l$ to $v_r$.  Let $y_{r,(j,l)}$ be a fractional
distribution on these states.  Define the color indicators
\[
 \xi_{u_r,i}=\sum_{l\ne i}y_{r,(i,l)},
 \qquad
 \xi_{v_r,i}=\sum_{j\ne i}y_{r,(j,i)}.
\]
For $a\in A$, put $\xi_{a,i}=x_{a,i}$. 

At $p_i^0=1/k$, define
\begin{align*}
 \Phi_i(x,y)
 =\frac{m_0}{k^3}
 +\frac1{k^2}\sum_{a\in A\cup P}\alpha_a\xi_{a,i}
 +\frac1k\sum_{ab\in\binom{A\cup P}{2}}\beta_{ab}\xi_{a,i}\xi_{b,i}
 +\sum_{abc\in E(\Gamma)}\xi_{a,i}\xi_{b,i}\xi_{c,i}.
\end{align*}
Then \eqref{eq:matched-atoms-deleted} implies that every term of $\Phi_i(x,y)$ meets each
block in at most one vertex.  Set
$x^0_{a,i}=1/k$ for $a\in A$ and $y^0_{r,\sigma}=\frac{1}{k(k-1)}$ for $r\in[t]$ and $\sigma\in\Sigma$. At $x^0,y^0$, we have $\Phi_i(x^0,y^0)=M(\cS)/k^3$ for every $i\in[k]$.

We first color most vertices of $A$ except at most $k(k-1)$ vertices.
 Let $P_i(x)\coloneqq \Phi_i(x,y^0)$ be a function of the variables
$\{x_{a,i}: a\in A\}$, and $F=\Gamma[A]$.  For a pair
$\{a,b\}\in\binom A2$, the coefficient in every color is
\[
        w_{ab}=\frac1k\beta_{ab}
        +\frac1k|\{z\in P:abz\in E(\Gamma)\}|,
\]
which is a nonnegative weighted graph. Then  Lemma~\ref{lem:graph-preserving} gives a fractional $k$-coloring $x'$ in which at most $k(k-1)$ vertices of $A$ are non-integral,
such that, for every $i\in[k]$,
\[
        E_{1,i}\coloneqq P_i(x')-P_i(x^0)\le 12(k-1)|A|^2\le 12(k-1)\varrho^2N^2,
\]
where the first inequality holds by $e(F)\le\binom{|A|}{3}$ and $\Delta_2(F)\le|A|$.

Next we color each block of $P$.
Fix $x'$ and put $\bar P_i(y)\coloneqq \Phi_i(x',y)$. We now apply
Lemma~\ref{lem:signed-rounding} to the block variables $y_{r,\sigma}$. 
For two distinct blocks $B_r,B_s$, the only way their variables can interact in $\Phi_i$ is through a weighted graph term $\beta_{uv}$, or through a 3-edge $uvw\in E(\Gamma)$.
Moreover, note that each $\xi_{z,i}$ is a sum of at most $k-1$ state variables $y_{r,\sigma}$.  Therefore
\[
 w_{rs}
 \le k\sum_{z\in B_r}\sum_{w\in B_s}
        \left(\frac{1}{k}\beta_{zw}(k-1)^2+d_\Gamma(zw)(k-1)^3\right)
 \le 4k^4\Lambda N.
\]
Further, we have
\begin{align*}    
        &W \le k\sum_{1\le r<s\le t}\sum_{z\in B_r}\sum_{w\in B_s}
        \left(\frac{1}{k}\beta_{zw}(k-1)^2+d_\Gamma(zw)(k-1)^3\right)\le 3k^4M(\cS)
\end{align*}
and $\Omega\le 4k^4\Lambda N$.
Similarly, the sum of coefficients of $y_{r,\sigma}$ over all $\sigma$ with $r$ fixed in $\Phi_i$ is at most
\[
        k(k-1)^3\bigl(D_1(u_r)+D_1(v_r)\bigr)
        \le 2k^4\Upsilon N^2,
\]
so $D\le 2k^4\Upsilon N^2$.  Lemma~\ref{lem:signed-rounding} with $q=k(k-1)$ gives an integral state assignment $y'$ such that, for every $i\in[k]$,
\[
        E_{2,i}\coloneqq \bar P_i(y')-\bar P_i(y^0)\le 4k^{10}\left(\sqrt{3M(\cS)\Lambda N}+\Upsilon N^2\right).
\]
Let $\xi'$ be the
fractional $k$-coloring of $A\cup P$ obtained from $x'$ and $y'$. Then, for every $i\in[k]$,
\[
        K_i(\cS;p^0,\xi')\le\frac{M(\cS)}{k^3}+12k\varrho^2N^2 +4k^{10}\left(\sqrt{3M(\cS)\Lambda N}+\Upsilon N^2\right),
\]
and at most $k(k-1)$ vertices, all in $A$, are non-integral.

Finally, we color the remaining vertices of $A\cup P$.
For every $i\in[k]$, put
\[
        T_i=\frac{M(\cS)}{k^3}+E_{1,i}+E_{2,i}-G_i(\Gamma;\xi')= L_i(\cS;p^0,\xi').
\]
Apply Lemma~\ref{lem:lower-rank-completion} to the configuration $\cS$, keeping every integral vertex fixed. There are an integral coloring $z$ of $A\cup P$ and $p\in S_k$ such that for each $i\in[k]$, $L_i(\cS;p,z)\le L_i(\cS;p^0,\xi')$.
It colors the final at most $k(k-1)$ remaining vertices of $A$ without increasing any $L_i(\cS;p^0,\xi')$. Now we consider how $G_i(\Gamma;\xi')$ changes for each $i$. 
Note that only edges incident with one of the final $k(k-1)$ vertices can increase $G_i(\Gamma;\xi')$.
Since $\Gamma$ is simple and $|A\cup P|\le\varrho N$, the total number of such edges is at most
$k(k-1)\varrho^2N^2$.  Let $A_i=\{u\in A\cup P:z_{u,i}=1\}.$
Then, for every $i\in[k]$,
\begin{align*}
 K_i(\cS;p,A_i)
 &\le T_i+G_i(\Gamma;z)\le\frac{M(\cS)}{k^3}
 +E_{1,i} +E_{2,i}+(G_i(\Gamma;z)-G_i(\Gamma;\xi'))\\
 &\le \frac{M(\cS)}{k^3}
 +(12k+k^2)\varrho^2N^2 +4k^{10}\left(\sqrt{3M(\cS)\Lambda N}+\Upsilon N^2\right).
\end{align*}
Moreover, since every block $B_r=\{u_r,v_r\}$ is in a state $(a,b)$ with $a\ne b$, $u_r,v_r$ must receive
distinct colors.
\end{proof}

\begin{lemma}\label{lem:matching-absorption}
For every fixed $k\ge2$ there is a constant $c_k>0$ such that the following
holds.  Let $\varrho,N,\lambda\ge1$, let $H$ be a $3$-graph on $V$ with $m\le N^3$ edges. Suppose $V=A\cup U$, where $|A|\le\varrho N,
        \Delta_H(U)\le N^2,
        \Delta_{2,H}(U)\le2\lambda N$.
If the graph on $U$ formed by pairs of codegree greater than $\lambda N$
contains a matching of size at least
$ \left\lceil\frac{c_kN}{\sqrt\lambda}\right\rceil,$
then $H$ has a $k$-partition
$V=V_1\cup\cdots\cup V_k$ in which, for every $i\in[k]$,
\[
        e_H(V_i)\le\frac m{k^3}+O(k^{10}(\varrho+k^{13})^2N^2).
\]
\end{lemma}

\begin{proof}[\bf Proof.]
Choose a matching $M$ of size $ t=\left\lceil\frac{c_kN}{\sqrt\lambda}\right\rceil,$  
where $c_k$ is a constant we choose later. 
Let $P=V(M)$ and $R=U\setminus P$, and let $\cK$ be the family of edges
containing some $B_r$ of $M$.  Since a $3$-edge cannot contain two
disjoint pairs, these edge families are disjoint and
\begin{equation}\label{eq:killed-mass}
 |\cK|=\sum_{uv\in M}d_H(u,v)
 >t\lambda N\ge c_k\sqrt\lambda\,N^2.
\end{equation}
We next apply Theorem~\ref{thm:forced-interface} and
Corollary~\ref{cor:residual-rounding} to obtain a $k$-partition of
$H$ in which the two vertices of every pair in $M$ lie in distinct
parts and each part contains sufficiently few edges.

Now delete $\cK$ from $H$ and put $H'=H\setminus\cK$.  Form the bad-set
configuration $\cS=(A\cup P;m_0,\alpha,\beta,\Gamma)$ with $m_0=e_{H'}(R)$, $\Gamma=H'[A\cup P]$ and two nonnegative functions
\begin{align*}
\alpha_a=|\{\{a,u,v\}\in E(H'):u,v\in R\}|,
 \qquad
 \beta_{ab}=|\{\{a,b,u\}\in E(H'):u\in R\}|.
\end{align*}
Then
\[
        M(\cS)=e(H')=m-|\mathcal K|\le N^3,
        \qquad
        |A\cup P|\le(\varrho+2c_k+2)N.
\]
For $z\in P$, we have $ D_1(z)=d_{H'}(z)\le d_H(z)\le N^2.$
If $z,w$ lie in two distinct pairs of $M$, then
\[
        D_2(zw)=d_{H'}(zw)\le d_H(zw)\le2\lambda N.
\]
Since we have deleted $\mathcal K$, $D_2(B_i)=0$ for every block $B_i$ of $M$. Since $M(\mathcal S)\le N^3$, Theorem~\ref{thm:forced-interface} with $\Upsilon=1$ and
$\Lambda=2\lambda$ gives a $k$-coloring $A\cup P=A_1\cup\cdots\cup A_k$ and $p\in S_k$ such that, for every $i\in[k]$,
\[
 K_i(\cS;p,A_i)
 \le\frac{m-|\mathcal K|}{k^3}+C_1(\varrho+2c_k+2)^2N^2+C_1\sqrt{2\lambda}\,N^2,
\]
where $C_1=4\sqrt6k^{10}$.

Next we construct the $k$-partition of $H$ corresponding to the $k$-coloring
$A\cup P=A_1\cup\cdots\cup A_k$. Note that, for distinct two vertices $r,s\in R$, $d_{H'}(r)\le N^2$, $d_{H'}(rs)\le2\lambda N$.
Corollary~\ref{cor:residual-rounding} extends the coloring of $A\cup P$ to
a $k$-partition $V=V_1\cup\cdots\cup V_k$ of $H'$ such that, for every $i\in[k]$,
\begin{align*}
 e_{H'}(V_i)
     &\le K_i(\cS;p,A_i)
 +C_2\left(\sqrt{M(\mathcal S)2\lambda N}+N^2\right)\\
     &\le K_i(\cS;p,A_i)
 +C_2\sqrt{2\lambda }N^2+C_2N^2,    
\end{align*}
where $C_2=2\sqrt3k^5$.
By \eqref{eq:killed-mass}, we have that, for every $i\in[k]$, 
\begin{align*}
     e_{H'}(V_i)
     &\le \frac{m-|\mathcal K|}{k^3}+C_{1}(\varrho+2c_k+2)^2N^2+C_1\sqrt{2\lambda}\,N^2+C_2\sqrt{2\lambda }N^2+C_2N^2\\
     &\le \frac{m}{k^3}+C_{1}(\varrho+2c_k+2)^2N^2+C_2N^2+\left(C_1\sqrt2+C_2\sqrt2-\frac{c_k}{k^3}\right)\sqrt\lambda\,N^2.
\end{align*}
Choose $c_k=( C_1+C_2+1)\sqrt2k^3=O(k^{13})$. Then we have that, for every $i\in[k]$,
\begin{align*}
     e_{H'}(V_i)
    \le \frac{m}{k^3}+O(k^{10}(\varrho+k^{13})^2N^2).
\end{align*}

Finally, every deleted edge in $\cK$ contains one of the bichromatic pairs of
$M$, so it is not contained in any part of the final $k$-partition.  Hence
$e_H(V_i)=e_{H'}(V_i)$ for every $i$, and the required bound follows.
\end{proof}

\section{Proof of the main theorem}
Now we are ready to prove the main theorem.
\begin{proof}[\bf Proof of Theorem~\ref{thm:main}]
The argument is trivial for $m=0$.  Put $N=m^{1/3}\ge1$ and define
\[
        A_0=\{v\in V(H):d_H(v)>N^2\},
        \qquad
        U_0=V(H)\setminus A_0.
\]
Since $\sum_vd_H(v)=3m=3N^3$, we have $|A_0|\le3N$.  Every vertex of $U_0$
has degree at most $N^2$, and hence every pair in $U_0$ has codegree at most
$N^2$.

Let $\lambda_j=2^j$ for $j=0,1,\ldots,J$, where $J$ is the least integer with
$\lambda_J\ge N$.  Let $c_k$ be the constant from
Lemma~\ref{lem:matching-absorption}.
We proceed from level $j=J$ down to stage $j=0$.
At some stage $j\ne0$, let $A$ be the current bad set and $U=V(H)\setminus A$ be the current good set. 
 Let $G_j$ be the graph on $U$ whose edges are the pairs of codegree greater than
$\lambda_{j-1}N$ in $H$, and let $M_j$ be a maximal matching in $G_j$. We check whether $|M_j|\ge \left\lceil c_kN/\sqrt{\lambda_{j-1}}\right\rceil$. If so, we stop. Otherwise, we move all vertices of $M_j$ to $A$ and proceed to stage $j-1$. 
This implies that, at every stage, the bad set $A$ satisfies
\begin{equation*}
 |A|
 \le3N+2c_kN\sum_{j=1}^{J}2^{-(j-1)/2}
 \le\varrho N,
\end{equation*}
where we put $\varrho=3+\frac{2c_k}{1-2^{-1/2}}=O(k^{13})$.
Moreover, since the vertices of a maximal matching form a vertex cover of $G_j$, every pair left in $U$ has codegree at most $\lambda_{j-1}N$.
Together with the fact that $d_H(u,v)\le N^2\le \lambda_{J}N$ for distinct $u,v\in U$ at the first stage $j=J$, this implies that at every stage $j$ we have $d_H(u,v)\le\lambda_{j}N$ for distinct $u,v\in U$.

If we stop at some stage $j\ne0$, that is, $|M_j|\ge \left\lceil c_kN/\sqrt{\lambda_{j-1}}\right\rceil$ at stage $j$, then applying Lemma~\ref{lem:matching-absorption} with $\lambda=\lambda_{j-1}$ gives a $k$-partition $V=V_1\cup\ldots\cup V_k$ such that for every $i\in[k]$,
\begin{align*}
    e_H(V_i)\le\frac m{k^3}+O(k^{10}(\varrho+k^{13})^2N^2)\le\frac m{k^3}+O_k(m^{2/3}).
\end{align*}
So we may suppose that we stop at stage $j=0$. Then the good set $U$ satisfies
$\Delta_H(U)\le N^2,
        \Delta_{2,H}(U)\le N$.
Applying Theorem~\ref{thm:low codegree case} gives a $k$-partition such that for every $i\in[k]$,
\begin{align*}
 e_H(V_i)
 &\le\frac m{k^3}
 +\left(\frac {k^2}2+12k\right)|A|^2+2k^5(\sqrt{3mN}+N^2)\\
 &\le\frac m{k^3}+O_k(N^2)
 =\frac m{k^3}+O_k(m^{2/3}).
\end{align*}
This proves the theorem.
\end{proof}

\section*{Declaration on the Use of AI}

The authors used ChatGPT 5.5 to assist in discussing proof strategies, checking proofs, and improving exposition.

\end{document}